\documentclass[11pt]{amsart}
\usepackage{amsmath,amsfonts,amsthm,amssymb,amscd, verbatim}
\def\classification#1{\def\@class{#1}}
\classification{\null}

\DeclareFontFamily{OT1}{rsfs}{}
\DeclareFontShape{OT1}{rsfs}{n}{it}{<-> rsfs10}{}
\DeclareMathAlphabet{\mathscr}{OT1}{rsfs}{n}{it}

\DeclareMathOperator{\GL}{GL}

\DeclareMathOperator{\SL}{SL}

\DeclareMathOperator{\vdeg}{\overrightarrow{\text{deg}}}

\DeclareMathOperator{\Sp}{Sp}
\DeclareMathOperator{\SO}{SO}

\DeclareMathOperator{\Tr}{Tr}
\DeclareMathOperator{\tr}{tr}
\DeclareMathOperator{\cl}{cl}
\newcommand{\pol}{pol}
\DeclareMathOperator{\Cl}{Cl}

\newcommand{\Z}{\mathbb{Z}}
\newcommand{\Aff}{\mathbb{A}}

\newcommand{\R}{\mathbb{R}}
\newcommand{\Kbar}{\overline{K}}

\newcommand{\Fp}{\mathbb{F}_p}

\newtheorem{prop}{Proposition}[section]

\newtheorem*{main}{Main Theorem}

\newtheorem{cor}[prop]{Corollary}
\newtheorem{lem}[prop]{Lemma}

\newcommand{\Aint}{A_{t, \overrightarrow{j}}}

\numberwithin{equation}{section}
\title{Growth of small generating sets in $\SL_n(\mathbb{Z}/p\mathbb{Z})$}
\author{Nick Gill and Harald Andr\'es Helfgott}
\address{School of Mathematics, University of Bristol, Bristol, BS8 1TW, United Kingdom}
\begin{document}

\begin{abstract}
Let $G=\SL_n$ and fix $\delta$ a positive number. We prove that there are positive numbers $\epsilon$ and $C$ such that, for all fields $K=\Z/p\Z$ ($p$ prime), and all sets $A\subset G(K)$ that generate $G(K)$, either $|A|<p^{n+1-\delta}$, $\delta>0$ or
$|A\cdot A\cdot  A|\geq C|A|^{1+\epsilon}$.
\end{abstract}
\maketitle

In \cite{helfgott2} and \cite{helfgott3}, the second author proved
that every set of generators of $\SL_2(\mathbb{F}_p)$ and
$\SL_3(\mathbb{F}_p)$ grows rapidly. This has since found numerous
applications (\cite{BG},\cite{BGS},\cite{V}) and generalisations to groups
of the same type ($\SL_2$, $\SL_3$) over other fields (\cite{MCC}, \cite{Di}).
In the present paper, 
we prove that small subsets of $\SL_n(\mathbb{F}_p)$ grow.

\begin{main}
Let $G=\SL_n$ and fix $\delta$ a positive number. Then there are positive numbers $\epsilon$ and $C$ such that, for all fields $K=\Z/p\Z$ ($p$ prime), and all sets $A\subset G(K)$ that generate $G(K)$, either $|A|<p^{n+1-\delta}$, $\delta>0$ or
$|A\cdot A\cdot  A|\geq C|A|^{1+\epsilon}$.
\end{main}
Here $|S|$ denotes the number of elements of a finite set $S$.

Our basic approach is to assume that $A$ is a set of generators for $G(K)$ that
does not grow; i.e., for a fixed positive $\epsilon$, $|A\cdot A\cdot A| \ll_n
|A|^{1+\epsilon}$. We will then show that if $A$ is ``small''
($|A|<p^{n+1-\delta}$) we have a
contradiction with a result from arithmetic combinatorics (Prop. \ref{p:
  vital})
related to the sum-product theorem and to incidence theorems.

\section{Notation and background}

We collect here the background ideas that will be needed in the sequel.

\subsection{Arithmetic combinatorics}
 Our notation in this area is standard and, in particular, is identical to that of \cite{helfgott3}. In this section $G$ is an arbitrary group.

Given a positive integer $r$ and a subset $S$ of a group $G$, we define
$$A_r=\{g_1\cdot g_2\cdots g_r \mid g_i\in A\cup A^{-1}\cup \{1\}\}.$$

Given real numbers $a,b,x_1,\dots, x_n$, we write 
$$a\ll_{x_1,\dots, x_n} b$$
to mean that the absolute value of $a$ is at most the real number $b$
multiplied by a constant $c$ depending only on $x_1,\dots, x_n$. In this
situation we also write $a=O_{x_1,\dots, x_n}(b)$. When we omit $x_1, \dots,
x_n$, and write $a\ll b$, we mean that the constant $c$ is absolute, unless 
we explicitly state otherwise.

Occasionally we may write vectors in place of $x_1,\dots, x_n$; for instance $a\ll_{\overrightarrow{v}} b$ where $a$ and $b$ are real and $\overrightarrow{v} = (v_1,\dots, v_n)\in \R^n$. In this case we mean that $a$ is at most the real number $b$ multiplied by a constant $c$ depending only on $v_1,\dots, v_n,$ and  on $n$.

The next result will be useful in its own right; it also gives some more context to the statement of the Main Theorem.

\begin{lem}\label{l: tripling}\cite[Lem.\ 2.2]{helfgott3}
{\rm (Tripling Lemma).} Let $k>2$ be an integer and let $c$ and $\epsilon$ be positive numbers. Then there exist positive numbers $c'$ and $\epsilon'$ such that, if $A$ is a finite subset of a group $G$ satisfying $|A_k|\geq c|A|^{1+\epsilon}$, then
 $$|A\cdot A\cdot A|\geq c'|A|^{1+\epsilon'}.$$
\end{lem}

We include one final piece of notation: For an $n$-tuple $(x_1,\dots, x_n)$, where $x_1,\dots, x_n$ are elements of some set, we write
$(x_1,\dots, \widehat{x_i}, \dots, x_n)$ to mean the $n-1$-tuple 
$$(x_1,\dots, x_{i-1},x_{i+1}, \dots, x_n).$$

\subsection{Degree and Escape}\label{s: escape}

We will use the notion of {\it escape from subvarieties} in a very similar way to in \cite{helfgott3}. A full explanation of this idea is given in \cite{helfgott3}; we note here only some essential definitions and results.

For $V$ an affine algebraic variety, we define $\dim V$ to be the dimension
of the irreducible subvariety of $V$ of largest dimension. As for the
degree, it will be best to see it as a vector: we define the {\em degree}
$\vdeg(V)$ of an arbitrary variety $V$ to be
\[(d_0,d_1,\dotsc,d_k),\]
where $k = \dim(V)$ and $d_j$ is the degree of the union of the irreducible components of $V$ of dimension $j$. 

If a regular map $\phi:V\mapsto W$ between two varieties 
$V\subset \mathbb{A}^m$, $W\subset \mathbb{A}^n$
is defined by polynomials $\phi_1, \phi_2,\dotsc , \phi_n$ on the variables
$x_1, x_2,\dotsc , x_m$, we define $\deg_{\pol}(\phi)$ to be 
$\max_j \deg(\phi_j)$.
(If several representations of $\phi$ by polynomials $\phi_1, \phi_2, \dotsc,
\phi_n$ are possible, we choose -- for the purposes of defining $\deg_{\pol}$ --
the one that gives us the least value of $\deg_{\pol}$.) 

Before proceeding we note an abuse of language: for a variety $V$ defined over a field $K$, and a subvariety $W$ defined over the algebraic completion $\Kbar$ of $K$, we will write $W(K)$ for $W(\Kbar)\cap V(K)$.

The  following series of results is related to the idea of escape and {\it non-singularity}.

 \begin{prop}\label{p: carbo}\cite[Prop. 4.1]{helfgott3}
Let $G$ be a group. Consider a linear representation of $G$
on a vector space $\mathbb{A}^n(K)$ over a field $K$.
Let $V$ be an affine subvariety of $\mathbb{A}^n$.

Let $A$ be a subset of $G$; let $\mathscr{O}$ be an $\langle A\rangle$-orbit
in $\mathbb{A}^n(K)$ not contained in $V$. Then there are constants $\eta>0$ and $m$
depending only on $\vdeg(V)$
such that,
for every $x\in \mathscr{O}$,
there are at least $\max(1, \eta |A|)$ elements $g\in A_m$ such that
$g x\notin V$.
\end{prop}

\begin{lem}\cite[Lem 4.3]{helfgott3}\label{l: ofor}%
Let $X\subseteq \Aff^{m_1}$ and $Y\subseteq\Aff^{m_2}$ be affine varieties defined over a field $K$. let $f:X\to Y$ be a regular map. Let $V$ be a subvariety of $X$ such that the derivative $Df|_{x=x_0}$ of $f$ at $x=x_0$ is a nonsingular linear map for all $x_0$ on $X$ outside $V$.

Then, for any $S\subseteq X(\Kbar)\backslash V(\Kbar)$, we have
$$|f(S)|\gg_{\overrightarrow{\deg}(X),\deg_{\rm pol}(f)} |S|.$$
\end{lem}

\begin{lem}\cite[Lem.'s 4.4 and 4.6]{helfgott3} \label{l: lemfac}
Let $n$ be a positive integer, $(a_1,\dots, a_d)$ and $(b_1,\dots, b_e)$ be tuples of non-negative integers. There exist $c,d>0$ and a positive integer $k$ such that
\begin{itemize}
\item for all fields $K$ such that $|K|>c$;
\item for all reductive algebraic groups $G\subset \GL_n$ defined over $K$ of degree $(a_1,\dots, a_e)$;
\item for all subvarieties $V$ of $G$ of degree $(b_1,\dots, b_e)$;
\item for all $A\subset G(K)$ which generate $G(K)$;
\end{itemize}
we have
\[|A_k \cap (G(K)\setminus V(K))| \geq d|A|.\]
\end{lem}

\begin{cor}\cite[Cor. 4.5]{helfgott3}\label{c: gotrol}
Let $G\subset \GL_n$ be an algebraic group of rank $r$ and $Y\subset \mathbb{A}^{m}$ 
an affine variety, both defined over a field $K$. Let $f:G\to Y$ be a regular map.
Let $V$ be a proper subvariety of $G$. Assume that 
the derivative $D f|_x$ of $f$ at $x$
is a nonsingular linear map for all $x$ on $G$ outside $V$.

Let $A\subset G(K)$ be a set
of generators of $G(K)$. Then
\[|f(A_k \cap (G(K)\backslash V(K)))|\gg_{\overrightarrow{\deg}(G), \overrightarrow{\deg}(V), \deg_{\rm pol}(f)} |A|,\]
where $k\ll_{\vdeg(V)} 1$.
\end{cor}

\begin{lem}\cite[Lem.\ 4.7]{helfgott3}\label{l: bull}
Let $G\subseteq GL_n$ be an algebraic group defined over a field $K$. Let $X/K$ and $Y$ be a affine varieties such that $\dim(Y)=\dim(G)$. Let $f:X\times G\to Y$ be a regular map. Define
$$f_x:G\to Y, \, g\mapsto f(x,g).$$
Then there is a subvariety $Z_{T\times G}\subset T\times G$ such that, for all $(x,g_0)\in (X\times G)(\Kbar)$, the derivative
$$(Df_x)|_{g=g_0}: (TG)|_{g=g_0} \to (TY)|_{f(x,g_0)}$$
is non-singular if and only if $(x,g_0)$ does not lie on $Z_{T\times G}$. Moreover,
$$\overrightarrow{\deg}(Z_{T\times G})\ll_{\overrightarrow{\deg}(T\times G), \deg_{\rm pol}(f), n} 1.$$
\end{lem}

\section{Known bounds for sets that do not grow.}\label{s: known}

We set $G=SL_n$, $K=\Z/p\Z$; we use a number of results about generating sets in $G(K)$ that ``do not grow'' \cite{helfgott3}. We start with information about how such a set intersects with a torus of $G$.

\begin{prop}\label{p: known1}\cite[Cor.'s 5.4 and 5.10]{helfgott3}
Let $G=SL_n$. 
\begin{enumerate}
\item There exists $C>0$ and $k$ a positive integer such that, for all finite fields $K$ of characteristic $\neq 2$, for any $T$ a maximal torus of $G$ defined over $\Kbar$, and for any $A\subset G(K)$ a set of generators of $G(K)$, we have
$|A\cap T(K)|\leq C|A_k|^{\frac1{n+1}}$.

\item There exist $d', \delta>0$ and $k'$ a positive integer such that, for all finite fields $K$ of characteristic $\neq 2$, and for any $A\subset G(K)$ a set of generators of $G(K)$ satisfying $|A\cdot A\cdot A|\leq d|A|^{1+\epsilon}$ for some $d,\epsilon>0$, there is a maximal torus $T$ such that $|A_{k'}\cap T(K)| \geq (d'd^\delta)|A|^{\frac1{n+1}-\delta\epsilon}$.
\end{enumerate}
\end{prop}

A torus $T$ in $G$ that satisfies the last inequality will be called {\it rich}; that is to say, such a $T$ has a large intersection with the set $A_k$.

For a set $B\subset G(K)$ define $\cl(B)$ to be the set of conjugacy classes of $G(K)$ that intersect $B$; if $B$ contains only one element, $g$, then write $\cl(g)$ for $\cl(B)$. We also write $B'$ for the set of semisimple elements in $B$; there are well-known bounds on $|\cl(A')|$ where $A$ is a generating set in $G(K)$ that does not grow.

\begin{prop}\label{p: known2}\cite[Cor.'s 5.7 and 5.11]{helfgott3}
Let $G=SL_n$.
\begin{enumerate}
\item There exists $C>0$ and a positive integer $k$ such that, for all finite fields $K$ of characteristic $\neq 2$, for any $T$ a maximal torus of $G$ defined over $\Kbar$, and for any $A\subset G(K)$ a set of generators of $G(K)$, we have
$|\cl((A_k)')|\geq C|A|^{\frac1{n+1}}$.

\item There exist $d', \delta>0$ and a positive integer $k'$ such that, for all finite fields $K$ of characteristic $\neq 2$, and for any $A\subset G(K)$ a set of generators of $G(K)$ satisfying $|A\cdot A\cdot A|\leq d|A|^{1+\epsilon}$ for some $d,\epsilon>0$, we have $|\cl(A')| \leq (d'd^\delta)|A_{k'}|^{\frac1{n+1}+\delta\epsilon}$.
\end{enumerate}
\end{prop}

Note that the Tripling Lemma (Lem.\ \ref{l: tripling}) implies that similar bounds apply to any set $A_l$ where $l\ll_n 1$; for instance, if $|A\cdot A \cdot A|\ll |A_l|^{1+\epsilon}$ and $l\ll_n1$, then $|A_l\cdot A_l \cdot A_l|\ll |A_l|^{1+O_l(\epsilon)}$, and so 
$$|\cl(A_l')| \ll_{l,n} |A_{kl}|^{\frac1{n+1}+O_{l,n}(\epsilon)}\ll_{l,n} |A|^{\frac1{n+1}+O_{l,n}(\epsilon)}.$$

We need to be sure that the intersection of $A_{k_\dagger}$ with a rich torus $T$ does not lie inside a subtorus (of a certain kind).

\begin{prop}\label{p: subtorus}\cite[Cor. 5.14]{helfgott3}
Let $G=SL_n$ and let $(a_1,\dots, a_r)$ be a tuple of non-negative integers. Then there exists $C>0$ and a positive integer $k$ such that for a finite field $K$ with characteristic $>n$; for $T$ a maximal torus of $G$ defined over $\Kbar$; for $\alpha:T\to\Aff^1$ a character of $T$ with kernel $T'$ satisfying $\deg(T')=(a_1,\dots, a_r)$; and for $A\subset G(K)$ a set of generators of $G(K)$, we have
$$|A\cap T'(K)|\leq C|A_k|^{\frac1{n+2}}.$$
\end{prop}

The significance of this special type of subtorus is explained by the next result.

\begin{lem}\cite[Lem.'s 4.2 and 4.14]{helfgott3}\label{l: alin}
Let $K$ be a field. Let $T/\overline{K} \subset (\GL_1)^n$ be a torus. Let $H$ be a subgroup of $T(\Kbar)$ that is contained in an algebraic subvariety $V$ of $T$ of positive codimension.

Then $H$ is contained in the kernel of a non-trivial character $\alpha:T\to GL_1$ whose exponents are bounded in terms of $n$ and $\vdeg(V)$ alone.
\end{lem}

\subsection{Further results on conjugacy}

We need several more ideas concerning conjugacy classes in $G$. An element $g$ of a linear algebraic group $G$ is said to be {\bf regular} if $\dim(C_G(g))$ is equal to the rank of $G$. When $G$ is connected, and $g$ is semisimple this is equivalent to requiring that $C_G(g)$ is a torus. When $G=SL_n$, an element is regular semisimple if and only if its eigenvalues are distinct.

\begin{lem}\cite[Lem.\ 5.9]{helfgott3}\label{l: regular}
 Let $G=SL_n$ and $K$ be a field. Then there is a subvariety $W/K$ of $G$ of positive codimension and degree $\vdeg(W)=(a_1,\dots, a_d)$ such that $a_1,\dots a_d, d<n^{2n}$, and every element $g\in G(K)$ not on $W$ is regular semisimple.
\end{lem}

Of course, since $W$ is a subvariety of $SL_n$ of positive codimension, we may apply the results of \S\ref{s: escape}, and escape from it. Escaping from a subvariety $W$ of a torus is only slightly harder, provided that the torus is rich.

\begin{lem}\label{l: torus regular}
Let $G=SL_n$ and let $(w_1,\dots, w_d)$ be a tuple of non-negative integers. There exists a positive integer $l$ and $\epsilon_0, c>0$ such that 
\begin{enumerate}
\item for all finite fields $K$ of characteristic $>n$; 
\item for any $A$ a set of generators for $G$ satisfying $|A\cdot A \cdot A|\leq d|A|^{1+\epsilon}$ for some $d,\epsilon>0$ with $\epsilon<\epsilon_0$;
\item\label{cuckoo} for any corresponding rich torus $T$ (so $|A_k\cap T(K)|\geq (d'd^\delta)|A|^{\frac1{n+1}-\delta\epsilon}$); 
\item for any $W$ a proper subvariety of $T$ of degree $(w_1,\dots, w_d)$;
\end{enumerate}
we have
$$|A_{l}\cap (T(K)\backslash W(K))|\geq c|A|^{\frac1{n+1}-\delta\epsilon}.$$
\end{lem}

Note that the value of $\delta$ given in the conclusion to Lem.\ \ref{l: torus regular} is fixed by Prop. \ref{p: known1} (as per the inequality written above at (\ref{cuckoo})).

\begin{proof}
Suppose first that $\langle A_l\cap T(K)\rangle$ is not contained in $W(K)$. Then apply Prop. \ref{p: carbo} with $A_l\cap T(K)$ instead of $A$, and $W$ instead of $V$, and the result follows.

Now suppose that $\langle A_l \cap T(K)\rangle \subseteq W(K)$. Then, by Lem.\ \ref{l: alin}, $\langle A_l \cap T(K)\rangle$ is contained in the kernel of a non-trivial character $\alpha:T\to \Aff^1$ whose exponents are bounded in terms of $n$ and $\vdeg(W)$ alone. Hence, by Prop. \ref{p: subtorus} with $A_l$ instead of $A$, and $l'$ instead of $k$,
$$|A_l\cap T(K)|\ll_{n, \vdeg(W)} |A_{ll'}|^{\frac1{n+2}} \ll |A|^{\frac1{n+2}+O_n(\epsilon)}.$$
This contradicts the assumption that $T$ is a rich torus.
\end{proof}

In what follows, we fix $\epsilon>0$, and fix $k_\dagger=\max\{k, k' ,l\}$ where $k$ and $k'$ are as given in Prop. \ref{p: known2}, and $l$ is as given in Lemma \ref{l: torus regular}; for the latter result we take $W$ to be the variety defined in Lemma \ref{l: regular}. Thus, for $A\subset G(K)$ a generating set that doesn't grow, and for $T$ a rich torus, we have tight bounds on the intersection of $A_{k_\dagger}$ with the set of regular elements in $T(K)$, as well as on the number of regular conjugacy classes in $A_{k_\dagger}$.

Our analysis of conjugacy classes in $G$ will require one more definition: Given a matrix $g$ in $\SL_n(K)$, we define
$\kappa(g)\in \mathbb{A}^{n-1}(K)$ to be the tuple
\[(a_{n-1},a_{n-2},\dotsc,a_1)\]
of coefficients of
\[\lambda^n + a_{n-1} \lambda^{n-1} + a_{n-2} \lambda^{n-2} + \dotsc +
a_1 \lambda + (-1)^{n} = \det(\lambda I - g) \in K\lbrack \lambda\rbrack\]
(the characteristic polynomial of $g$).

As is well-known, $\kappa(g) = \kappa(h g h^{-1})$ for any $h$, i.e.,
$\kappa(g)$ is invariant under conjugation. If $g$ is a regular semisimple
element of $\SL_n$ then $\kappa(g)$ actually determines the conjugacy class
$\cl(g)$ of $g$. 

\section{A new bound on sets that do not grow}

Just as in Helfgott's work on $SL_3$, the trace is going to be important for us, although we will use it in a very different way.

\subsection{Independence of traces}

In this section we show that a certain set of linear forms is linearly dependent, and that any subset is linearly independent. 

Let $m\in K[t_1,\dots, t_n]$ be a monomial; we call $m$ {\it square-free} if $t_j^2$ does not divide $m$ for any $j=1,\dots, n$.

\begin{lem}\footnote{We thank an anonymous referee for significantly simplifying the proof of this result.}\label{l: vander}
Let $s_j$, $j=1,\dots, n$, be distinct elements of $K^*$. Then, for any $i$, $0\leq i\leq n$,
$$A_i=\left(\begin{matrix} 
1 & s_1 & \cdots & s_1^{i-1} & s_1^{i+1} & \cdots & s_1^{n} \\
1 & s_2 & \cdots & s_2^{i-1} & s_2^{i+1} & \cdots & s_2^{n} \\
\vdots & \vdots &  & \vdots & \vdots & & \vdots \\
1 & s_n & \cdots & s_n^{i-1} & s_n^{i+1} & \cdots & s_n^{n} \\
\end{matrix}\right)$$
has determinant equal to
$$\left(\prod_{1\leq j<k\leq n}(s_k-s_j)\right)\times S_{n-i},$$
where $S_{n-i}$ is the sum of all square-free monomials of degree $n-i$ in the variables $s_1,\dots, s_n$.
\end{lem}
\begin{proof}
We prove the result by calculating a Vandermonde determinant in $\{s_1,\dots, s_n, t\}$ for some variable $t$. We define
$$f= \det\left(\begin{matrix} 
1 & s_1 & \cdots & s_1^{n} \\
1 & s_2 & \cdots & s_2^{n} \\
\vdots & \vdots & & \vdots \\
1 & s_n & \cdots & s_n^{n} \\
1 & t & \cdots & t^n \\
\end{matrix}  \right)$$
and consider $f$ as a polynomial of degree $n$ in $t$. Observe that $\det A_i = (-1)^{n+i} c_i$ where $c_i$ is the coefficient of $t^i$ in $f$.

The classical formula for a Vandermonde determinant implies that
$$f=\prod\limits_{1\leq j<k\leq n} (s_k-s_j) \prod\limits_{1\leq k\leq n}(t-s_k).$$
Then
$f= \left(\prod\limits_{1\leq j<k\leq n} (s_k-s_j)\right)g$ where $g = \prod\limits_{1\leq k\leq n}(t-s_k)$ is a polynomial of degree $n$ in $t$. Write $d_i$ for the coefficient of $t^i$ in $g$; to prove the lemma it is sufficient to observe that $d_i=(-1)^{n-i}S_{n-i}$.
\end{proof}

\begin{lem}\label{l: lindep}\footnote{An anonymous referee pointed out that this result follows directly from the Cayley-Hamilton identity; c.f. the proof of Lemma \ref{l: f one}.}
Suppose that $t$ is a regular element of $T(K)$. For $i=0,\dots, n$, define linear forms,
$$l_i: G(K)\to K: g \mapsto \tr(t^i g).$$
These forms are linearly dependent. Furthermore, there exists a subvariety $W$ in $T$ of positive codimension and degree $\ll_n 1$, such that if $t\in T(K)\backslash W(K)$, then any subset of $n$ of these linear forms is linearly independent. 
\end{lem}
\begin{proof}
Write $T(K)$ as the group of $n\times n$ diagonal matrices over $K$ of determinant 1. The linear forms $l_i$ are forms on $n$ variables, $g_{11}, \dots, g_{nn}$. Since there are $n+1$ of these, they must be linearly dependent.

Now consider a set, $L$, of $n$ of these linear forms. Since $t$ is regular, all eigenvalues are distinct. Thus $L$ will be linearly independent provided a matrix of the form given in Lemma \ref{l: vander} is non-zero. 

For $i,j=1,\dots, n$ with $i<j$ define $V_{i,j}$ to be the variety defined by $g_{ii}- g_{jj} = 0$, and let $V_{-1}=\bigcup\limits_{1\leq i<j\leq n} V_{i,j}$. Thus $V(K)$ contains all elements of $T(K)$ that are not regular. Next define $W_i$ to be the variety defined by the equation $S_i=0$, where $S_i$ is as in Lemma \ref{l: vander}; let $W_{-1}=\cup_{i=0}^n W_i$ and let $W=V_{-1}\cup W_{-1}$. 

Clearly $V_{i,j}, i,j=1,\dots, n$ with $i<j$, and $W_i, i=0,\dots, n$ are subvarieties of $T$ of positive codimension. Then Lemma \ref{l: vander} yields the result.
\end{proof}

\subsection{A new bound}

Now we work towards proving a new bound on generating sets that ``do not grow.'' We begin with some notation. Fix $t$ an element of a torus $T(K)$ in $G(K)$, $i$ some integer satisfying $0\leq i\leq n$. For $g$ an element of $G(K)$, define two $n$-tuples as follows:
\begin{equation*}
 \begin{aligned}
\Tr_{\widehat{i},t}(g)&=(\tr(g), \tr(tg), \dots, \widehat{\tr(t^ig)}, \dots, \tr(t^n g)); \\
\Cl_{\widehat{i},t}(g)&=(\cl(g), \cl(tg), \dots, \widehat{\cl(t^ig)}, \dots, \cl(t^n g)). 
\end{aligned}
\end{equation*}

For a set $A\subset G(K)$, define
\begin{equation*}
 \begin{aligned}
\Tr_{\widehat{i},t}(A)&=\{\Tr_{\widehat{i},t}(g) \, \mid \, g\in A\, \}; \\
\Cl_{\widehat{i},t}(A)&=\{\Cl_{\widehat{i,t}}(g) \, \mid \, g\in A\, \}; \\
\Cl_{\widehat{i},t}'(A)&=\{\Cl_{\widehat{i},t}(g) \, \mid \, g\in A, \textrm{ and } g,tg,\dots, t^ng \textrm{ are semisimple} \, \}.
 \end{aligned}
\end{equation*}

The inequality that we will prove relates to $\Cl_{\widehat{i},t}(A)$. In fact $\Cl_{\widehat{i},t}(A)$ corresponds (generically) to $\cl_{T(K)}(A)$, the set of $T(K)$-conjugacy classes of $G(K)$ that intersect $A$. We will not need this correspondence in what follows.

\begin{lem}\label{l: dosh}
Let $G=SL_n$, and let $K$ be a field. For $h_0, t_0,\dots, t_n \in G(\Kbar)$, define 
$$f_{h_0,t_0,\dots, t_n}: G\to \Aff^{n^2-1}: g\mapsto (\kappa(h_0g), \kappa(t_0g),\dots, \kappa(t_ng)).$$
Define
$$f_{\widehat{i}, h_0, t_0,\dots, t_n}: G\to \Aff^{n^2-n}: g\mapsto (\kappa(h_0g), \kappa(t_0g), \dots,\widehat{\kappa(t_i g)}, \dots \kappa(t_ng)).$$
Let $T/\Kbar$ be a maximal torus of $G$. Then there exists $h_0\in G(\Kbar)$, $t\in T(\Kbar)$, and $g_0\in G(\Kbar)$, such that the derivative of $f_{\widehat{i}, h_0, t, t^2\dots, t^n}$ at $g=g_0$ is a nonsingular linear map.
\end{lem}

This result should be compared with \cite[Lemma 5.5]{helfgott3}; the method of proof is very similar.

\begin{proof}
We may write the elements of $G(\overline{K})$ so that the elements of
$T(\overline{K})\subset G(\overline{K})$ become diagonal matrices.
Let
\[g_0 = \left(\begin{matrix}
0 & 1 & 0 & \dotsb &0 & 0\\
0 & 0 & 1 & \dotsb &0 & 0\\
\vdots & \vdots & \vdots & \vdots & \vdots & \vdots\\
0 & 0 & 0 & \dotsb &1 & 0\\
0 & 0 & 0 & \dotsb &0 & 1\\
(-1)^{n-1} & 0 & 0 & \dotsb & 0 & 0\end{matrix}\right).\]
Let $\vec{r} = (r_1,r_2,\dotsc,r_n)$ be a vector in $\overline{K}^n$
with $r_1 \cdot r_2 \dotsb r_n = 1$. Define
\begin{equation}\label{eq:ororh}
h_1 = \left(\begin{matrix}
r_{1} & 0 & \dotsc & 0\\
0 & r_{2} & \dotsc & 0\\
\vdots & \vdots & \vdots & \vdots\\
0 & 0 & \dotsc & r_{n}\end{matrix}\right).
\end{equation}

Let us look, then, at the derivative at $g=I$ of $g\mapsto
\kappa(h_1^j g_0 g)$ for $0\leq j\leq n$. The derivative 
at $g=I$ of the map taking $g$ to the coefficient of $\lambda^{n-1}$ in
$\det(\lambda I - h_1^j g_0 g)$ (i.e., to $(-1)$ times the trace of $h_1^j
g_0 g$)
is equal to the map taking each matrix $\gamma$ in the tangent space
$\mathfrak{g}$ to $G$ at the origin to 
\[(-1)\cdot (r_{1}^j \gamma_{2,1} + r_{2}^j \gamma_{3,2} + \dotsc + r_{k}^j \gamma_{k+1,k} +
\dotsc + (-1)^{n-1} r_{n}^j \gamma_{1,n} ),
\]
where we write $\gamma_{k,l}$ for the entries of the matrix $\gamma$.

The derivative at $g=I$
of the map taking $g$ to the coefficient of $\lambda^{n-2}$
in $\det(\lambda I - h_1^j g_0 g)$ is the map taking each $\gamma$ in
$\mathfrak{g}$ to 
\[\begin{aligned}
(-1)\cdot (r_{1}^j r_{2}^j \gamma_{3,1} &+ r_{2}^j r_{3}^j \gamma_{4,2} + \dotsc +
r_{k}^j r_{k+1}^k \gamma_{k+2,k} + \dotsc + 
r_{n-3}^j r_{n-2}^j \gamma_{n-1,n-3} \\
&+ 
r_{n-2}^j r_{n-1}^j \gamma_{n,n-2} + r_{n-1}^j\cdot (-1)^{n-1} r_n^j \gamma_{1,n-1} +
(-1)^{n-1} r_{n}^j r_1^j \gamma_{2,n}).
\end{aligned}\]
In general, for $1\leq l\leq n-1$, the derivative at $g=I$ of the map taking
$g$ to the coefficient of $\lambda^{n-l}$ in $\det(\lambda I - h_1^j g_0 g)$ is the map taking $\gamma$ to
\begin{equation}\label{eq:prese}(-1)
\sum_{k=1}^{n-k} 
(r_{k}^j \cdot r_{\underline{k+1}}^j \dotsb r_{\underline{k+l-1}}^j) \cdot
\gamma_{\underline{k+l}, k} +
(-1)^{n} \sum_{k=n-l+1}^{n} 
(r_{k}^j \cdot r_{\underline{k+1}}^j \dotsb r_{\underline{k+l-1}}^j) \cdot
\gamma_{\underline{k+l}, k},
\end{equation}
where by $\underline{a}$ we mean the only element of $\{1,2,\dotsc,n\}$
congruent to $a$ modulo $n$.

We see that the entries of $\gamma$ present in (\ref{eq:prese})
are disjoint for distinct $1\leq l\leq n-1$ (and disjoint from $\{\gamma_{1,1},
\gamma_{2,2},\dotsc,\gamma_{n,n}\}$, which would appear for $l=0$).
Now, for $l$ fixed, (\ref{eq:prese}) gives us a linear form on $n$ variables
$\gamma_{\underline{k+l},k}$ for each $\leq j\leq n$. Let us
check that, for every $1\leq l\leq n-1$, any $n$ of these linear forms are linearly independent, provided that $\vec{r}$ was chosen correctly.

This is the same as checking that the $n-1$ determinants
\begin{equation}\label{eq:odorem}
\left|(r_{k}^j\cdot r_{\underline{k+1}}^j\cdots r_{\underline{k+l-1}}^j)\right|_{1\leq k\leq n, 0\leq j\leq n , j\neq i}
\end{equation}
for $1\leq l\leq n-1$ are non-zero for some choice of
$r_1,r_2,\dotsc, r_n$ with $r_1\cdot r_2 \dotsb r_n = 1$.
(What we really want to check is that the determinant (\ref{eq:odorem})
is non-zero after all signs in 
some columns are flipped; since those flips do not affect the
absolute value of the determinant, it is just as good to check that
the determinant (\ref{eq:odorem}) itself is non-zero.)

These are generalised Vandermonde determinants of the type studied in Lemma \ref{l: vander}; that result implies that they have absolute value equal to
\[
\prod_{k_1<k_2} (r_{k_2}\cdot r_{\underline{k_2+1}}
\dotsb r_{\underline{k_2+l-1}} -
r_{k_1}\cdot r_{\underline{k_1+1}}
\dotsb r_{\underline{k_1+l-1}}) \times S_{n-i},
\]
where $S_{n-i}$ is the sum of all square-free monomials $m$ of degree $n-i$ in the variables $r_{k}\cdot r_{\underline{k+1}}
\dotsb r_{\underline{k+l-1}}$ for $k=1,\dots, n$. 

For any given $k_1, k_2\in\{1,\dots, n\}$ with $k_1\ne k_2$, and $l\in\{1,\dots, n-1\}$, there exist
$r_1, r_2,\dotsc,r_n\in \overline{K}$ with $r_1 r_2 \dotsb r_n = 1$
such that
$r_{k_1}\cdot r_{\underline{k_1+1}}
\dotsb r_{\underline{k_1+l-1}} \ne
r_{k_2}\cdot r_{\underline{k_2+1}}
\dotsb r_{\underline{k_2+l-1}}$. Thus,
$r_{k_1}\cdot r_{\underline{k_1+1}}
\dotsb r_{\underline{k_1+l-1}} =
r_{k_2}\cdot r_{\underline{k_2+1}}
\dotsb r_{\underline{k_2+l-1}}$ defines a subvariety $W_{l,k_1,k_2}$ of
positive codimension in the (irreducible) variety $V\subset \mathbb{A}^n$ of
all tuples $(r_1,r_2,\dotsb,r_n)$ with $r_1 r_2 \dotsb r_n = 1$.

In addition, for $i,j=1,\dots, n$ with $i<j$, define $V_{i,j}$, the subvariety of $V$ given by the equation $r_i=r_i$; clearly the set $V\backslash V_{i,j}(\Kbar)$ is non-empty, and so these varieties are also of positive codimension in $V$. Then 
$$W = \left(\bigcup\limits_{1\leq k_1,k_2\leq n,\; k_1 \ne k_2 \atop 1\leq l\leq n-1} W_{l,k_1,k_2}\right)\bigcup\left(\bigcup\limits_{1\leq i<j\leq n} V_{i,j}\right)$$
is a finite union of subvarieties of $V$ of positive codimension.
Take $\vec{r}$ to be any point of $V(\overline{K})$ outside
$W(\overline{K})$.
It remains to choose
$h_0$
so that the derivative of
\[g\mapsto \kappa(h_0 g)\]
at $g = I$ is a linear map of full rank on the diagonal entries
$\gamma_{1,1},\gamma_{2,2}\dotsc,\gamma_{n-1,n-1}$ of $\mathfrak{g}$.
Let
\begin{equation}\label{eq:babyl}h_0 = \left(\begin{matrix}
t_{1} & 0 & \dotsc & 0\\
0 & t_{2} & \dotsc & 0\\
\vdots & \vdots & \vdots & \vdots\\
0 & 0 & \dotsc & t_{n}\end{matrix}\right),\end{equation}
where $t_1, t_2,\dotsc,t_n \in \overline{K}$ fulfil
$t_1 t_2 \dotsb t_n = 1$.
Then the derivative at $g=I$ of the map taking $g$ to the coefficient
of $\lambda^{n-1}$ in $\det(\lambda I - h_0 g)$ (i.e., to $(-1)$ times
the trace of $h_0 g$) equals the map taking $\gamma$ to
\[(-1) \cdot (
t_1 \gamma_{1,1} + t_2 \gamma_{2,2} + \dotsb + t_n \gamma_{n,n}).\]
In general, the derivative of the map taking $g$ to the coefficient
of $\lambda^{n-k}$ ($1\leq k\leq n-1$) in $\det(\lambda I - h_0 g)$
equals the map taking $\gamma$ to 
\[(-1)^k\cdot (c_{k,1} \gamma_{1,1} + c_{k,2} \gamma_{2,2} + \dotsb 
+ c_{k,n} \gamma_{n,n}),\]
where $c_{k,i}$ is the sum of all monomials
$t_{j_1} t_{j_2} \dotsc t_{j_k}$, $1\leq j_1<j_2<\dotsb<j_k\leq n$,
such that one of the indices $j_l$ equals $i$. (For example,
$c_{2,1} = t_1 \cdot (t_2 + t_3 + \dotsb + t_n)$.) 
Thus, our task is to find
for which $t_1, t_2,\dotsc, t_n$ the determinant
\[|c_{i,j} - c_{i,n}|_{1\leq i,j\leq n-1}\]
is non-zero.
Clearly, this will happen precisely when
\[|c_{i-1,j}|_{1\leq i,j\leq n}\ne 0,\]
where we adopt the (sensible) convention that $c_{0,j} = 1$ for all $j$.

A brief computation gives us that
\[|c_{i-1,j}|_{1\leq i,j\leq n} =
(-1)^{\lfloor n/2\rfloor} \cdot |t_j^{i-1}|_{1\leq i,j\leq n}.\]
This is a Vandermonde determinant; it equals $\prod_{j_1<j_2}
(t_{j_2} - t_{j_1})$. The equation $t_{j_2} = t_{j_1}$ defines a subvariety
of positive codimension in the variety $V\subset \mathbb{A}^n$ of all
$t_1,t_2,\dotsc,t_n$ with $t_1 t_2 \dotsb t_n = 1$. Thus, we may choose
$t_1,t_2,\dotsc,t_n$ such that $t_1 t_2 \dotsb t_n = 1$ and
$\prod_{j_1 < j_2} (t_{j_2} - t_{j_1}) \ne 0$.
\end{proof}

We are now able to prove the proposition that we have been aiming for (c.f. \cite[Prop.\,5.15]{helfgott3}).

\begin{prop}\label{p: conj tuple}
Let $G = \SL_n$. There exists a positive integer $k$ and $c,e>0$ such that for any $(a_1,\dots, a_d)$, a tuple of non-negative integers, there exists a positive integer $k'$ and $c'>0$ such that
\begin{itemize}
\item for any $K$ a finite field such that $|K|<e$;
\item for any $W/\overline{K}$, a proper subvariety of $G$, such that $\vdeg(W)=(a_1,\dots, a_d)$;
\item for any $E\subset T(K)$ where $T/\overline{K}$ is a maximal torus of $G$.
\item for any $A\subset G(K)$, a set of generators of $G(K)$;
; 
\end{itemize}
one of the following holds
\begin{enumerate}
\item\label{it:goro} there exists $h_0\in A_k$, and a subset
$E'\subset E_k$ with $|E'|\geq c|E|$ such that, for each $t\in E'$,
there are $\geq c'|A|$ distinct tuples
\[(\kappa(h_0 g), \kappa(g), \kappa(t g), \kappa(t^2 g),\dots, \widehat{\kappa(t^i g)}, \dots, \kappa(t^n g))
\in \mathbb{A}^{n^2-1}(K)\]
with $g\in A_{k'}$ satisfying $h_0g,
g, t g, t^2 g, \dots, t^n g \notin W(K)$, or
\item\label{it:gara} $E$ is contained in the kernel of a non-trivial character
$\alpha:T\to GL_1$ whose exponents are  bounded in terms of $n$ alone.
\end{enumerate}
\end{prop}
\begin{proof}
Let $X = G\times T$ and $Y = (\mathbb{A}^{n-1})^{n+1} = 
\mathbb{A}^{n^2-1}$. Let $f:X\times G\to Y$
be given by
\[f((h_0,t),g) = (\kappa(h g),\kappa(g), \kappa(t g), \kappa(t^2 g),\dotsc, \widehat{\kappa(t^i g)}, \dotsc,
\kappa(t^n g)).\]
Let $Z_{X\times G}$ be as in Lemma \ref{l:  bull}; then
$\vdeg(Z_{X\times G})\ll_n 1$. Thanks to Lem.\ \ref{l:  dosh}, we know
$Z_{X\times G}$ is a proper subvariety of $X\times G$. 

Let $Z_{G\times T\times G}$ be $Z_{X\times G}$ under the identification
$G\times T\times G = X\times G$; write the elements of $Z_{G\times T\times G}$
in the form $(h_0,t,g)$. By the argument in \cite[\S2.5.3]{helfgott3},
there is a proper subvariety $Z_G\subset G$ (with $\vdeg(Z_G)\ll_{\vdeg(Z_{G\times T\times G})} 1$, and so $\vdeg(Z_G)\ll_n 1$) such that, 
for all $h_0\in G(\overline{K})$ not on $Z_G$, 
the fibre $(Z_{G\times T\times G})_{h=h_0}$ is a proper subvariety of
 $T\times G$.

By escape (Lem.\ \ref{l:  lemfac};
it is here that that $|K|\gg_n 1$ is used),
there is an $h_0\in A_k$, $k\ll_n 1$, such that $h_0$ lies outside $Z_G$;
thus, by the definition of $Z_G$, the fibre
$V_{T\times G}:=(Z_{G\times T\times G})_{h=h_0}$ is a proper subvariety of
 $T\times G$. Again by \cite[\S2.5.3]{helfgott3}, there is a proper
subvariety $V_T$ with $\vdeg(V_T)\ll_{\vdeg(V_{T\times G})} 1$ (and so
$\vdeg(V_T)\ll_n 1$)
 such that, for all $t_0\in T(\overline{K})$ not on $V_T$,
the fibre $(V_{T\times G})_{t = t_0}$ is a proper subvariety of $G$.

Suppose first that $\langle E\rangle\not\subset V_T(K)$. We may then
use escape from subvarieties (Prop. \ref{p: carbo} with
$A = E$, $V = V_T(K)$ and $G = \mathscr{O} = \langle E\rangle$) to
obtain a subset $E'\subset E_k$ ($k\ll_n 1$) with
$|E'|\gg_n |E|$ and $E'\subset T(K)\setminus V_T(K)$. Now consider any
 $t_0\in E'$. 
The fibre $(V_{T\times G})_{t = t_0}$ is a proper subvariety of $G$,
and, since $W$ is a proper subvariety of $G$, we conclude that
\[V' = (V_{T\times G})_{t = t_0} \cup h_0^{-1} W \cup W \cup t_0^{-1} W
\cup \dotsc \cup t_0^{-n} W.\]
is a proper subvariety of $G$ as well (with $\vdeg(V')\ll_{n,\vdeg(W)} 1$).
We now recall the definition of $Z_{X\times G}$ (a variety outside
which the map $f$ is non-singular) and
 use the result on non-singularity
(Cor. \ref{c:   gotrol} applied to the function 
$f_{h_0,t_0}:G\to Y$ given by
$f_{h_0,t_0}(g)=f((h_0,t_0),g)$) to obtain that
\[|f_{h_0,t_0}(A_{k'} \cap (G(K)\setminus V'(K)))| \gg_{n,\vdeg(W)} |A|\]
with $k'\ll_{n,\vdeg(W)} 1$. This gives us conclusion (\ref{it:goro}).

Suppose now that $\langle E\rangle\subset V_T(K)$. By Lemma \ref{l: alin}, we obtain that $E$ is contained in the kernel
of a non-trivial character $\alpha:T\to \mathbb{A}^1$ whose exponents are
$\ll_n 1$.
\end{proof}

Given our generating set $A$ satisfying $|A\cdot A\cdot A|\ll_n |A|^{1+\epsilon}$, recall that Prop. \ref{p: known1} implies that there exists a torus $T$ that is {\it rich}. In other words
$$|A_k\cap T(K)|\gg |A|^{\frac1{n+1}-O_n(\epsilon)}$$
for some $k\ll_n 1$. In fact, by Lem. \ref{l: torus regular} we know that when $T$ is rich, the set of regular elements in $A_{k_\dagger}\cap T(K)$ has size $\gg |A|^{\frac1{n+1}-O_n(\epsilon)}$, where $k_\dagger \ll_n 1$. These facts allow us to prove the following corollary.

\begin{cor}\label{c: dot}
Let $\epsilon>0$ and suppose that $\epsilon$ is smaller than some constant $\epsilon_0$ depending only on $n$. Let $A$ be a generating set such that $|A\cdot A\cdot A|\ll_n |A|^{1+\epsilon},$ and let $T$ be a rich maximal torus. Then there are integers $k,l\ll_n k_\dagger$ and a set $E'$ in $A_k \cap T(K)$ such that $|E'|\gg |A|^{\frac1{n+1}+O_n(\epsilon)}$ and
$$|A|^{\frac{n}{n+1}-O_n(\epsilon)}\ll_n |\Cl_{\widehat{i}, t}'(A_l)|,$$
for all $t\in E'$. Moreover there are $\gg |A|$ elements of $A_l$ such that $g,tg, \dots, t^ng$ are semisimple.
 \end{cor}


\begin{proof}
Let $W$ be the variety of elements in $G$ that are not semisimple, and note that, by Lem. \ref{l: regular}, $\vdeg(W)\ll_n 1$. Set $E=A_{k_{\dagger\dagger}}\cap T(K)$ where $T$ is a rich torus and $k_{\dagger\dagger}=\max\{k_\dagger, k\}$ for $k$ the constant given in Prop. \ref{p: conj tuple}; thus Prop. \ref{p: known1} implies that $|E|\gg_n |A|^{\frac1{n+1}-O_n(\epsilon)}$. Now apply Prop. \ref{p: conj tuple} to $A$. Since we are assuming that $|A\cdot A\cdot A|\ll_n |A|^{1+\epsilon}$, we may assume that $K$ is larger than some constant depending only on $n$ (otherwise our assumption is trivially violated). Thus we lie in situation (a) or (b) of Prop. \ref{p: conj tuple}.

 Suppose first that we are in situation (b). Then Lem. \ref{l: alin} and Prop. \ref{p: subtorus} imply that $|E|\ll_n |A_l|^{\frac1{n+2}}$ where $l\ll_n 1$. For $\epsilon$ small enough with respect to $n$ this gives a contradiction.

Thus we are in situation (a) and observe first that
$$|E'|\gg_n |E| \gg_n |A|^{\frac1{n+1}-O_n(\epsilon)}$$
as required. Now Prop. \ref{p: known2} implies that that there are at most $|A_{k_{\dagger\dagger}}|^{\frac1{n+1}+O_n(\epsilon)}\ll_n|A|^{\frac1{n+1}+O_n(\epsilon)}$ distinct tuples $\kappa(g)$, for $g$ a semisimple element in $A$. Indeed the same upper bound holds for tuples $\kappa (h_0 g)$, for $g\in A$ such that $h_0g$ is simple. Prop. \ref{p: conj tuple} implies, therefore, that there are $\gg_n |A|^{1-\frac1{n+1}-O(\epsilon)}$ elements in $\Cl_{\widehat{i},t}'(A)$ as required.

The final statement of the corollary follows immediately from Prop. \ref{p: conj tuple}.
\end{proof}

\section{Moving from conjugacy classes to traces}

Cor. \ref{c: dot} gives us a bound on the size of the set $\Cl_{\widehat{i}, t}'(A_l)$. In order to prove that small sets grow we will need to work with the tuple $\Tr_{\widehat{i}, t}'(A_l)$. To make the transfer from conjugacy classes to traces we introduce the concept of {\it wealth}. Fix $l_\dagger$, the value for $l$ given by Cor. \ref{c: dot}; $t$ a regular semisimple element; and $i$ an element of $\{0,\dots, n\}$. The {\bf wealth} of an element $r$ of $K$ is defined to be
$$\diamondsuit_{t,i}(r) = |\{\cl(g): \tr(t^ig)= r, g\in A_{l_\dagger} \textrm{ and } t^ig \textrm{ is semisimple} \, \}|.$$

Now fix $\epsilon$  a small positive number; for $t\in A_{k_\dagger}\cap T(K)$ we define $A_{t,(j_0, j_1,\dots, j_n)}$ where $(j_0, \dots, j_n)$ is a tuple of integers:
\begin{equation*}
\begin{aligned}
A_{t, (j_0, j_1,\dots, j_n)} = \{g\in A_{l_\dagger}\, \mid \, &2^{j_0} < \diamondsuit_{t,0}(\tr(g))\leq 2^{j_0+1}, \\
&2^{j_1}< \diamondsuit_{t,1}(\tr(tg))\leq 2^{j_1+1}, \dots, \\
&2^{j_n}< \diamondsuit_{t,n}(\tr(t^ng))\leq 2^{j_n+1}, \textrm { and } \\
&g, tg, \dots, t^ng \textrm{ are semisimple}\}. 
\end{aligned}
\end{equation*}
Thus $A_{t, (j_0, j_1,\dots, j_n)}$ is a set of elements in $A_{l_\dagger}$ such that the corresponding tuple of traces of each element represents a roughly constant number of conjugacy classes. We often write $\Aint$ for $A_{t, (j_0, \dots, j_n)}$.

Note that this definition is dependent on our generating set $A$, on the constant $\epsilon$, and on the element $t$; note, furthermore, that there are at most $(\log_2|A|)^{n+1}$ tuples $(j_0, \dots, j_n)$ such that $A_{t, (j_0, j_1,\dots, j_n)}$ is non-empty. 

\begin{lem}\label{l: wealth}
Suppose that $T$ is a rich torus and take $t\in T(K)$. Suppose that $|\Aint|\geq |A|^{1-O(\epsilon)}$. Then
$$2^{\max_ij_i - \min_i j_i} \ll d|A|^{O(\epsilon)}$$
provided that $t$ lies outside a proper subvariety $W\subset T$ of degree $\ll_n 1$.
\end{lem}

Note that the subvariety $W$ will be defined so as to include all of the non-regular elements.

\begin{proof}
By Cor. \ref{c: dot}, the number of different tuples $\Tr_{\widehat{i}, t}(g)$, as $g$ varies over $\Aint$, is at least
$$\frac{\Cl_{\widehat{i}, t}'(A)}{2^{j_0+\cdots +j_{i-1}+j_{i+1}+\cdots+j_n+n}}\gg_n \frac{|A|^{\frac{n}{n+1}-O(\epsilon)}}{2^{j_0+\dots+j_{i-1}+j_{i+1}+\cdots+j_n+n}}.$$
On the other hand
$$|\{\tr(t^ig) \, \mid \, g\in\Aint, 2^{j_i} <\diamondsuit_{t,i}(\tr(t^ig))\leq 2^{j_i+1}\}|\leq \frac{|\cl(A_{i+l_\dagger}')|}{2^{j_i}}.$$
So, by Prop. \ref{p: known2} the number of different tuples $\Tr_{\widehat{i},t}(g)$ can take on is at most
$$\frac{|\cl(A_{n+l_\dagger}')|^n}{2^{j_0+\dots+j_{i-1}+j_{i+1}+\cdots+j_n}}\ll_n \frac{|A|^{\frac{n}{n+1}+O_n(\epsilon)}}{2^{j_0+\dots+j_{i-1}+j_{i+1}+\cdots+j_n}}.$$

Now we need a definition:
$$\Tr_{\widehat{i}, \widehat{k},t}(g)=(\tr(g), \tr(tg), \dots, \widehat{\tr(t^ig)}, \dots, \widehat{\tr(t^kg)}, \dots, \tr(t^n g)).$$
Again the number of different tuples $\Tr_{\widehat{i}, \widehat{k},t}(g)$ as $g$ varies over $\Aint$ is at most
$$\frac{|\cl(A_{nk+l_\dagger}')|^{n-1}}{2^{j_0+\dots+j_{i-1}+j_{i+1}+\cdots+j_{k-1}+j_{k+1}+\cdots+j_n}}\ll_n \frac{|A|^{\frac{n-1}{n+1}+O_n(\epsilon)}}{2^{j_0+\dots+j_{i-1}+j_{i+1}+\cdots+j_{k-1}+j_{k+1}+\cdots+j_n}}.$$
 
Then, for a given tuple $\overrightarrow{r}=(r_0,\dots, r_{i-1}, r_{i+1}, \dots, r_{k-1},r_{k+1}, \dots, r_{n})\in K^{n-1}$ such that $r_i=\tr(t^ig)$ for some $g\in\Aint$, and a given index $k$, there are at least
\begin{equation}\label{e: lower}
 \begin{aligned}
& \gg_n\left( \frac{|A|^{\frac{n}{n+1}-O(\epsilon)}}{2^{j_0+\dots+j_{i-1}+j_{i+1}+\cdots+j_n}}\right)\times \left(\frac{|A|^{\frac{n-1}{n+1}+O(\epsilon)}}{2^{j_0+\dots+j_{i-1}+j_{i+1}+\cdots+j_{k-1}+j_{k+1}+\cdots+j_n}}\right)^{-1}  \\
&= \frac{|A|^{\frac1{n+1}- O(\epsilon)}}{2^{j_k}} 
 \end{aligned}
\end{equation}
values for $\tr(t^kg)$ such that $\Tr_{\widehat{i},t}(g) = (r_0,\dots, r_{i-1}, r_{i+1}, \dots, r_{k-1},\tr(t^kg), r_{k+1}, \dots, r_n)$ for some $g\in A_{l_\dagger}$. Now we wish to apply Lem. \ref{l: lindep}; to do this we must assume, as we do, that $t$ lies outside an exceptional subvariety $W$ of bounded degree. Then Lem. \ref{l: lindep} implies that  there exist $a_1,\dots, a_{k-1},a_{k+1},\dots, a_n\in K$ such that, for all $g\in G(K)$,
$$\tr(t^k g) = a_1\tr(g) + \cdots a_{k-1}\tr(t^{k-1}g) + a_{k+1}\tr(t^{k+1}g)+\cdots a_n\tr(t^n g).$$
We conclude, therefore, that (\ref{e: lower}) is a lower bound on values for $\tr(t^ig)$ such that $$\Tr_{\widehat{k},t}(g) = (r_1,\dots, r_{i-1}, \tr(t^ig), r_{i+1}, \dots, r_{k-1}, r_{k+1}, \dots, r_{n+1}).$$

Observe that there are at least $2^{j_i}$ values of $\cl(t^ig)$ for a fixed value of $\tr(t^ig)$. Thus
$$|\cl(A)| \geq 2^{j_i}\frac{|A|^{\frac1{n+1}- O(\epsilon)}}{2^{j_k}}.$$
Since $|\cl(A)|\ll |A|^{\frac1{n+1}+O(\epsilon)}$ we conclude that $2^{j_i-j_k}\ll|A|^{O(\epsilon)}$ as required.
\end{proof}

\section{Small sets grow}

We continue to analyse our set $A$ that ``doesn't grow''. Let $T$ be a rich torus; by the pigeonhole principle,  every $t\in T(K)$ has a ``popular tuple'', $\overrightarrow{j}$, such that $|\Aint|\geq |A|^{1-O(\epsilon)}$. In this case Lem. \ref{l: wealth} implies that $\max_i j_i - \min_ij_i \ll O(\epsilon)\log_2|A|.$ Define an integer
$$p_{\overrightarrow{j}}=|\{t\in T(K)\, \mid \, t\not\in W(K) \textrm{ and }|\Aint| \geq |A|^{1-O(\epsilon)}\}|$$
where $W$ is as in Lem. \ref{l: wealth}. 

Lem. \ref{l: lindep} implies that, for $t\in T(K)\backslash W(K)$, the linear forms
$$l_i: G(K)\to K, \, g\mapsto \tr(t^i g)$$
are linearly dependent; what is more any subset of $n$ of these is linearly independent. These facts allow us to define
$$f:T(K)\backslash W(K)\to K^n, \, t \mapsto \overrightarrow{r}=(r_0, \dots, r_{n-1})$$
where $\overrightarrow{r}$ satisfies the equation
$$\tr(t^ng) = r_0\tr(g) + r_1\tr(t g) + \cdots + r_{n-1}\tr(t^{n-1} g)$$
for all $g\in G$. 

We need to know that $f$ is, in some sense, close to being one-to-one. The following lemma gives the result that we need.\footnote{We thank an anonymous referee for significantly simplifying the proof of this result.}

\begin{lem}\label{l: f one}
Suppose that $S$ is a subset of $T(\Kbar)\backslash W(\Kbar)$, where
$W$ is as above. Then
$$|f(S)|\geq \frac1{n!} |S|.$$
\end{lem}
\begin{proof}
Recall the definition of $\kappa(t)=(a_1, \dots, a_{n-1})$ given at the end of Section \ref{s: known}. We apply the Cayley-Hamilton theorem (which states that a square matrix $t$ satisfies the polynomial $\det(\lambda I - t)=0$) to conclude that 
\begin{equation*}
\begin{aligned}
& \det(\lambda I - t) = \lambda^n + a_{n-1}\lambda^{n-1} + \cdots + a_1\lambda + (-1)^n; \\
\Rightarrow & t^n = -a_{n-1}t^{n-1} -a_{n-2}t^{n-2} - \cdots - a_1 t - (-1)^n; \\
\Rightarrow & t^ng = -a_{n-1}t^{n-1}g -a_{n-2}t^{n-2}g - \cdots - a_1 tg - (-1)^ng \textrm{ for all } g\in G(K); \\
\Rightarrow & \tr(t^ng) = -a_{n-1}\tr(t^{n-1}g)- \cdots - a_1 \tr(tg) - (-1)^n\tr(g) \textrm{ for all } g\in G(K); \\
\Rightarrow & r_{n-1}=-a_{n-1}, r_{n-2}=-a_{n-2}, \dots, r_1=-a_1, r_0=(-1)^{n+1}. 
\end{aligned}
\end{equation*}
Now, all elements in $T(\Kbar)\backslash W(\Kbar)$ are regular; write $T(\Kbar)$ as the set of diagonal matrices in $G(\Kbar)$. Then two elements $g, h\in T(\Kbar)\backslash W(\Kbar)$ satisfy $\kappa(g)=\kappa(h)$ if and only if there is an element $\sigma$ in the symmetric group on $n$ letters such that 
$g_{ii}=h_{\sigma(i)\sigma(i)}$. The result follows.
\end{proof}

Now set $\overrightarrow{j^{\rm max}} = (j^{\max}_0, \dots, j^{\max}_n)$ be the tuple for which $p_{\overrightarrow{j}}$ is maximal (so $\overrightarrow{j^{\rm max}}$ is the ``most popular'' popular tuple). Define
$$D=\{t\in T(K)\cap A_{k_\dagger}\, \mid \, t\not\in V(K)\cup W(K) \textrm{ and }|A_{t, \overrightarrow{j^{\rm max}}}|\geq |A|^{1-O(\epsilon)}\}.$$

Observe that $|D|=p_{\overrightarrow{j^{\rm max}}}$. Since there are at most $(\log_2|A|)^n$ choices for $\overrightarrow{j^{\rm max}}$, Lem. \ref{l: torus regular} implies that $|D|\gg_n |A|^{\frac1{n+1}-O(\epsilon)}$.

We also define ${j^{\rm max}_{\max}}$ (resp. ${j^{\rm max}_{\min}}$) to be the maximum (resp. minimum) element in $\{j^{\max}_0, \dots, j^{\max}_n\}.$

We need the following result from arithmetic combinatorics; it is a strengthening of \cite[Cor. 3.8]{helfgott3}.

\begin{prop}\label{p: vital}
There exists $c>0$ such that, for any $n>2$ a positive integer and any $\delta>0$, there exist $e,\eta>0$ such that
\begin{itemize}
\item for every prime $p$, and field $K=\Z/p\Z$;
\item for every set $X\subset K$ such that $|X|\leq p^{1-\delta}$;
\item for every set $Y\subset K^n$;
\item for each $\overrightarrow{y}\in Y$, and $X_{\overrightarrow{y}}\subseteq X^n$ such that
$$\overrightarrow{y}\cdot X_{\overrightarrow{y}}=\{y_1x_1+\cdots + y_nx_n: \overrightarrow{x}\in X_{\overrightarrow{y}}\}$$
is contained in $X$; 
\end{itemize}
either $|Y|<|X|^{cn\eta}$ or
\begin{equation}\label{e: conc}
|X_{\overrightarrow{y}}|\leq e|X|^{n-\eta}\textrm{ for some } \overrightarrow{y}\in Y.
\end{equation}
\end{prop}
\begin{proof}
Suppose $|Y|\geq |X|^{c n \eta}$ ($c>0$ to be chosen later) and 
$|X_{\vec{y}}|\gg |X|^{n-\eta}$ for every $\vec{y} \in Y$. We will
try to derive a contradiction.

There is a coordinate $1\leq i\leq n$ such that the image  $\pi_i(Y)$
of $Y$ under the projection $\pi_i$ to the $i$th coordinate has
$\geq |Y|^{1/n}$ elements. We can assume without loss of generality that
$n=1$. Let $Y'\subset Y$ be a maximal set such that $\pi_1|_{Y'}$ is 
injective. Then $|Y'|\gg |Y|^{1/n}$.

For $x_3,x_4,\dotsc,x_n\in X$, let \[R_{x_3,x_4,\dotsc,x_n} =
\{(\vec{y}, x_1, x_2) \in Y'\times X\times X:
 (x_1,x_2,\dotsc,x_n)\in X_{\vec{y}}\}.\]
By assumption, $\sum_{x_3,\dotsc,x_n\in X} |R_{x_3,\dotsc,x_n}| \gg
|Y'| |X|^{n-\eta}$, and so we can fix $x_3,x_4,\dotsc,x_n\in X$
such that $S=R_{x_3,\dotsc,x_n}$ has $\gg |Y'| |X|^{2-\eta}$ elements.
For $\vec{y} \in Y'$, let $c_{\vec{y}} = x_3 y_3 + \dotsc + x_n y_n$.

We have $S\subset Y'\times X\times X$; let $S_{\vec{y}} = \{
(x_1,x_2)\in X\times X: (\vec{y},x_1,x_2) \in S\}$. There is an $\epsilon>0$
such
that the sets $S_{\vec{y}}$ with $|S_{\vec{y}}|\leq \epsilon |X|^{2-\eta}$
contribute at most $\frac{1}{2} |S|$ elements to $S$. Let $Y''$ be the set
of all $\vec{y} \in Y'$ with $|S_{\vec{y}}|> \epsilon |X|^{2-\eta}$. Then
$\sum_{\vec{y}\in Y} |S_{\vec{y}}| \geq \frac{1}{2} |S| \gg |Y'| |X|^{2-\eta}$
and so $|Y''|\gg |Y'| |X|^{\eta} \gg |X|^{(c-1) \eta}$. By the definition
of $S_{\vec{y}}$,
\[\{x_1 y_1 + x_2 y_2 + c_{\vec{y}} : (x_1,x_2)\in S_{\vec{y}}\}\subset X.\]
Thus, for all $\vec{y} \in Y''$, there are $> \epsilon |X|^{2-\eta}$
distinct pairs $(x,x') = (x_1,x_1y_1+x_2 y_2 +c_{\vec{y}})\in X\times X$ 
such that $x'- y_1 x \in y_2 X + c_{\vec{y}}$. Thus the {\em energy}
$E_+(X,y_1 X) = \sum_{d\in X - y_1 X} |\{a\in X, b\in y_1 X: a- b = d\}|^2$
is greater than $\frac{1}{|y_2 X + c_{\vec{y}}|} \left(
\epsilon |X|^{2-\eta}\right)^2
= \epsilon^2 |X|^{3-2\eta}$. Hence
\begin{equation}\label{eq:cosou}
\sum_{y_1\in \pi(Y'')} E(X,y_1 X) > \epsilon^2 \pi_1(Y'') |X|^{3-2\eta} .
\end{equation}
Recall that $\pi_1|_{Y'}$ is injective and $Y''\subset Y'$ has $\gg
|X|^{(c-1)\eta}$ elements; thus $|\pi_1(Y'')|\gg |X|^{(c-1)\eta}$.

Hence, for $c$ and $p$
large enough and $\eta$ smaller than a constant times $\delta$,
 (\ref{eq:cosou}) is contradiction  with Theorem C of \cite{Bo}. (If $p$
is smaller than a constant, the statement we are trying to prove is trivially true.)
\end{proof}

We are now able to prove our main result.

\begin{main}
Let $G=\SL_n$ and fix $\delta$ a positive number. Then there are positive numbers $\epsilon$ and $C$ such that, for all fields $K=\Z/p\Z$ ($p$ prime), and all sets $A\subset G(K)$ that generate $G(K)$, either $|A|<p^{n+1-\delta}$, $\delta>0$ or
$|A\cdot A\cdot  A|\geq C|A|^{1+\epsilon}$.
\end{main}

\begin{proof}
Suppose that $|A|< p^{n+1-\delta}$ and that $A$ does not grow; we seek a contradiction to Prop. \ref{p: vital}. Let $\eta$ be some (small) positive number.

First set $R=K=\Z/p\Z$. We define $X$ to be the following set:
$$\{\tr(t^ia) \, \mid \, t\in D, a\in A_{t, \overrightarrow{j^{\max}}}, i=0,\dots, n; \textrm{ and } t^ia \textrm{ is semisimple } \}.$$
Observe that $|X|\leq \frac{|\cl((A_{nk_\dagger+l_\dagger})')|}{2^{j^{\max}_{\min}}}$ where $(A_{nk_\dagger+l_\dagger})'$ is the set of semisimple elements in $A_{nk_\dagger+l_\dagger}$. Prop. \ref{p: known2} implies that  $|\cl(A_{nk_\dagger+l_\dagger})'|\ll_n |A|^{\frac{1}{n+1}+O(\epsilon)}$, and so 
$$|X|\leq \frac{|A|^{\frac{1}{n+1}+O(\epsilon)}}{2^{j^{\max}_{\min}}}.$$
 Thus if we pick $\epsilon$ small enough we know that $|X|<p^{1-\delta'}$ for some positive number $\delta'$, and we have satisfied the assumption in Prop. \ref{p: vital}.

Now define $Y=f(D)$ and for $\overrightarrow{y}\in Y$ we define $X_{\overrightarrow{y}}$ to be the set of elements $\overrightarrow{x}$ in $X^n$ such that $\overrightarrow{y}\cdot \overrightarrow{x}$ is equal to the trace of a semisimple element in $DA$, i.e. such that $\overrightarrow{y}\cdot \overrightarrow{x}$ lies in $X$. 

Our first job is to show that $|X_{\overrightarrow{y}}|\gg_n |X|^{n-\eta}$ for all $\overrightarrow{y}\in Y$. Take $\overrightarrow{y}=f(t)$ for some $t\in D$. Then $\overrightarrow{y}=(r_0, \dots, r_{n-1})$
where $\overrightarrow{r}$  satisfies the equation
$$\tr(t^ng) = r_0\tr(g) + r_1\tr(t g) + \cdots + r_{n-1}\tr(t^{n-1} g)$$
for all $g\in G$.

Now Prop. \ref{p: conj tuple} implies that, for all $i=0,\dots, n$, 
$$|\Cl_{\widehat{i},t}(A_{t, \overrightarrow{j^{\max}}})|\gg_n |A_{t, \overrightarrow{j^{\max}}}|^{\frac{n}{n+1}+O(\epsilon)} \gg_n {|A|^{\frac{n}{n+1}+O(\epsilon)}}.$$
Thus we conclude that
$$|\Tr_{\widehat{i},t}(A_{t, \overrightarrow{j^{\max}}})|\gg \frac{|A|^{\frac{n}{n+1}+O(\epsilon)}}{2^{j^{\max}_0+ \cdots + j^{\max}_n+n}}.$$

Next observe that Lem. \ref{l: wealth} implies that the wealth for each entry in the tuple differs only by a factor of $O(\epsilon)$, hence we are able to conclude that 
$$|\Tr_{\widehat{i},t}(A)|\gg \frac{|A|^{\frac{n}{n+1}-O(\epsilon)}}{(2^{j^{\max}_{\min}})^n}\gg |X|^{n-O(\epsilon)} .$$
If we pick $\epsilon$ to be small enough in terms of $\eta$ then this implies that
$|\Tr_{\widehat{i},t}(A)|\gg|X|^{n-\eta},$ as required.

Finally we need to show that $|Y|\gg |X|^{O(n\eta)}$. Now Lem. \ref{l: f one} implies that 
$$|Y|=|f(D)|\gg_n |A|^{\frac1{n+1}+O_n(\epsilon)}.$$
We conclude that, provided we choose $\epsilon$ small enough with respect to $n$ and $\eta$, then $|Y|\gg_n |X|^\eta.$ Thus we have a contradiction with Prop. \ref{p: vital} as required.
\end{proof}

\section{Generalisations}

We give a brief explanation as to how ideas from the current paper may be applied more generally. It should be a relatively straightforward matter to strengthen the Main Theorem to the following statement: for every $\delta>0$ there is an $\epsilon>0$ such that if $|A|<p^{n+1-\delta}$ or if $p^{n+1+\delta}<|A|<p^{2n+2-\delta}$, then
$$|A\cdot A\cdot A|\gg_n |A|^{1+\epsilon},$$
where $\epsilon$ and the implied constant depend only on $n$.

To prove this statement, our current proof would need to be extended to consider vectors of the form
$$\left(\begin{array}{c}
         \tr(g) \\ \tr(g^{-1})
        \end{array}\right)$$
for $g\in A$; analysis of the set of all such vectors would lead to a contradiction of a result similar to Prop. \ref{p: vital}, but over the ring $K\times K$, rather than over $K$. This would be similar to the approach adopted in \cite[\S 6]{helfgott3}, but would also need a concept of {\it wealth} for it to work in the current setting.

As it is we have restricted our attention simply to the set of traces, $\{\tr(g) | g\in A\}$. We did this because things are more transparent in this setting; but more importantly we did this because there seems no obvious way of extending the ideas in this paper to cover sets of size greater than $p^{2n+2-\delta}$. 

The significance of $\tr(g)$ and $\tr(g^{-1})$ is that these are entries in the tuple $\kappa(g)$ (the first and last entry respectively). Both of these entries correspond to the fundamental $n$-dimensional representations of $G$ over $\Kbar$; when $n=3$ these are all of the fundamental representations of $G$ and hence an analysis of these two representations gives information about sets of size all the way up to $|G|^{1-\delta}$.

When $n\geq4$, however, the other entries in $\kappa$ correspond to representations of dimension strictly greater than $n$. This fact prevents us from using them as we use $\tr(g)$ and $\tr(g^{-1})$.

The methods in the present paper seem to extend easily to $\Sp_{2 n}$,
$n\geq 2$; there they show that $|A\cdot A \cdot A| \geq |A|^{1+\epsilon}$
for every $A\subset \Sp_{2n}(K)$ with $|A|<p^{2n+1-\delta}$,
$\delta>0$.
 However, there is a technical reason why the extension of
these methods to $\SO_n$
would be more problematic: both $\SL_n$ and $\Sp_{2 n}$ obey
$(\dim(\pi) + 1)\cdot \dim(T) \leq \dim(G)$, where $\pi$ is the
smallest-dimensional faithful representation of $G$ and $T$ is the maximal
torus
of $G$. In contrast, in the case of $\SO_n$, $(\dim(\pi)+ 1)\cdot \dim(T)
= (2 m +1 )\cdot m > \frac{n(n-1)}{2} = \dim(G)$
if $n = 2m$, and $(\dim(\pi)+ 1)\cdot \dim(T)
= (2 m +2)\cdot m > \frac{n(n-1)}{2} = \dim(G)$ if $n=2m+1$.

\begin{center}
* * *
\end{center}

The difficulties in extending the result to $\SO_n$ complicate the following 
approach to proving that every subset $A\subset \SL_n(K)$ grows (as usual, $K=\Fp$ here).
Let $A\subset \SL_n(\mathbb{F}_p)$, $|A|\geq p^{n-1+\delta}$, $\delta>0$ and consider the natural action of $SL_n(K)$ on $V=K^n$.
Let $H$ be a maximal parabolic subgroup of $\SL_n(\mathbb{F}_p)$ stabilizing a one-dimensional subspace $\langle\vec{v}\rangle$ of $V$, i.e.,
\[H = \{g\in \SL_n(\mathbb{F}_p) : \text{$g \vec{v} = \lambda \vec{v}$ for
some $\lambda \in K^*$}\}.\]
Then $\lbrack G:H\rbrack \ll p^{n-1}$, and so (by \cite[Lemma 7.2]{helfgott3}),
$|A^{-1} A \cap H| \gg p^{\delta}$. Then one can try to show that 
$A^{-1} A \cap H$ grows in $H$; from this, it would follow that $A$
itself grows (\cite[Lemma 7.3]{helfgott3}). The problem here is that we could
have $\langle A^{-1} A \cap H\rangle \sim \SO_{n-1}$, 
for example. This possibility could
be avoided if we could find a rich torus $T$ in $H$, or, conversely, if we
could choose $H$ so that it contains a rich torus $T$.

The reason why this would be enough is the following. Suppose $\langle A^{-1} A \cap H \rangle \sim \SO_{n-1}$,
or even $\langle A^{-1} A \cap H \rangle \sim K^* \times \SO_{n-1}$. 
The maximal torus of
$K^* \times \SO^{n-1}$ has smaller dimension than that of $\SL_n$. Hence,
since $A^{-1} A \cap T(K)$ is contained in $A^{-1} A \cap H
\subset \langle A^{-1} A \cap H\rangle \sim K^* \times \SO_{n-1}$,
 we cannot possibly have 
$|A_k \cap T(K)|
\gg |A|^{\dim(T)/\dim(G)}$ (see Prop.\ 2.3 in the present paper), i.e.,
$T$ cannot be rich. 

In general, by the same argument, 
if $T$ is rich and contained in $H$, then $A^{-1} A \cap H$
cannot be contained in any algebraic subgroup of $H$ of bounded degree
whose maximal tori are
 of smaller dimension than $T$. The only connected algebraic subgroups
$H'$ of $H$ whose maximal tori are of the same dimension as a maximal
torus of $\SL_n$ have form
$$U\rtimes((\SL_{m_1}\times\dotsc\times\SL_{m_k})T),$$
where $U$ is a unipotent group, $m_1+m_2+\dotsc + m_k=n$, and $T$ is a maximal torus.

It follows that $A^{-1}A \cap H$ is contained in one such group $H'$,
but not in any proper algebraic subgroup of $H'$ of bounded degree.
By \cite[Prop.\ 4.2]{helfgott3}, we conclude that $A^{-1} A\cap H$
is not contained in any algebraic subvariety of $H'$ of bounded degree.
Thus, we can work as if $A^{-1} A \cap H$ generated $H'$. By induction on $n$,
the theorem $|A\cdot A\cdot A|\geq |A|^{1+\epsilon}$ would then be proven for
all $A\subset \SL_n(\mathbb{F}_p)$ of arbitrary size $|A|\leq 
|\SL_n(\mathbb{F}_p)|^{1-\delta}$, $\delta>0$.

The only gap in the above argument is that one must still prove that there
exists a maximal parabolic subgroup $H$ containing a rich torus $T$. If a 
rich torus $T$ fixes some vector $\vec{v}\in K^n$, then $T$ lies in a maximal
parabolic subgroup $H_{\vec{v}}$
of the required type. The only problem remaining,
then, is to find a rich torus $T$ (i.e., a torus $T$ with $|A_k\cap
T(K)| \gg |A|^{\frac{\dim T}{\dim G} - O(\epsilon)}$) such that $T$ has at
least
one eigenvector in $K^n$ (as opposed to just in $\overline{K}^n$).

\begin{center}
* * *
\end{center}

One possible approach to this last question is to look for regular semisimple 
elements in $A^{-1} A \cap H$, where $H$ is a maximal parabolic subgroup.
If there is one such element, there will be many, and hence there will be
a rich maximal torus $T$ inside $H$. If instead there are many non-semisimple
elements, one can conjugate them by coset representatives of $H$ to produce
$\gg |A|$ non-semisimple elements of $A_k$; this, in turn, results in
$|A\cdot A\cdot A|\geq |A|^{1+\epsilon}$ by an argument sticking varieties in
different directions (as in \cite[\S 4.3]{helfgott3}). The only problem is
that $A^{-1} A \cap H$ may consist almost entirely
 of non-regular semisimple elements, as then too many of their
conjugates could be identical;
a different argument would be needed in this case.

\subsection{Final note}
Shortly after the completion of this paper, and well after the completion of
the research leading to it, Pyber and Szabo \cite{PS}, and Breuillard, Green,
and Tao \cite{BGT} independently announced that they had proven results more
general than those in the present paper. According to these announcements, their methods are based in part on those of the second author \cite{helfgott2, helfgott3}, on which the present work also rests.

{\bf Funding.}
This work was supported by
EPSRC (Engineering and Physical Sciences Research Council) grant EP/E054919/1.

{\bf Acknowledgements.} 
We would like to thank R. Hill, M. Kassabov and N. Nikolov for interesting
conversations on subgroups and tori in $\SL_n$. In addition,
K. Raghavan, A. Zalesski and S. Murray made useful comments.
 We are likewise grateful
for the hospitality of the people of Orsay (Paris Sud XI).

\end{document}